\documentclass{article}
\usepackage[utf8]{inputenc}
\usepackage{geometry}
\usepackage{float}
\usepackage{caption}
\usepackage{graphicx}
\usepackage{xcolor}
\usepackage{amsmath}
\allowdisplaybreaks
\usepackage{amsthm}
\usepackage{amssymb}
\usepackage{mathrsfs}
\usepackage{amsfonts}
\usepackage{lineno}
\usepackage{fancyhdr}
\usepackage{dsfont}
\usepackage{indentfirst}
\usepackage{hyperref}     
\usepackage{cleveref}
\usepackage{enumitem}
\usepackage{aliascnt} 
\usepackage{tikz}  
\usetikzlibrary{shapes,arrows.meta,positioning}
\usepackage[normalem]{ulem}
\allowdisplaybreaks   
\newtheorem{claim}{Claim}

\newtheorem*{claim*}{Claim}
\newenvironment{clproof}{\begin{list}{}{%
              \setlength{\leftmargin}{3mm}%
              } \item {\it Proof.} }{\hfill$\lozenge$\end{list}}
\newtheorem{theorem}{Theorem}[section]

\newtheorem*{THMMAIN}{Theorem~\ref{main:path}}
\newtheorem*{THMMAIN2}{Theorem~\ref{main:star}}
\newtheorem*{THMMAIN3}{Theorem~\ref{main:cycle}}
\newaliascnt{lemma}{theorem}
\newtheorem{lemma}[lemma]{Lemma}
\aliascntresetthe{lemma}
\Crefname{lemma}{Lemma}{Lemmas}
\Crefname{lemma}{Lemma}{Lemmas}

\newaliascnt{corollary}{theorem}

\aliascntresetthe{corollary}
\Crefname{corollary}{Corollary}{Corollaries}
\Crefname{corollary}{Corollary}{Corollaries}

\newaliascnt{proposition}{theorem}
\newtheorem{proposition}[proposition]{Proposition}
\aliascntresetthe{proposition}
\Crefname{proposition}{Proposition}{Propositions}
\Crefname{proposition}{Proposition}{Propositions}

\newaliascnt{observation}{theorem}

\aliascntresetthe{observation}
\Crefname{observation}{Observation}{Observations}
\Crefname{observation}{Observation}{Observations}

\newaliascnt{example}{theorem}

\aliascntresetthe{example}
\Crefname{example}{Example}{Examples}
\Crefname{example}{Example}{Examples}

\newaliascnt{question}{theorem}
\newtheorem{question}[question]{Question}
\aliascntresetthe{question}
\Crefname{question}{Question}{Questions}
\Crefname{question}{Question}{Questions}

\theoremstyle{definition}  
\newaliascnt{remark}{theorem}
\newtheorem{remark}[remark]{Remark}
\aliascntresetthe{remark}
\Crefname{remark}{Remark}{Remarks}
\Crefname{remark}{Remark}{Remarks}

\newcommand{\rw}{\mathrm{rwSAT}}
\newcommand{\s}{\mathrm{rwsat}}

\newcommand{\ex}{{\rm ex}}

\title{Weak rainbow saturation numbers of paths, stars and cycles}
  \author{Jiawen Bo\thanks{School of Mathematical Sciences and LPMC, Nankai University, Tianjin 300071, P.R.
 China. Email: bojiawen@mail.nankai.edu.cn.}, Xiaopan Lian\thanks{  Center for Combinatorics and LPMC, Nankai University, Tianjin 300071, P.R.
  China. Email: Lian@nankai.edu.cn.},~Jianing Liu\thanks{School of Mathematical Sciences and LPMC, Nankai University, Tianjin 300071, P.R.
 China. Email: liujianing@mail.nankai.edu.cn.}}

\begin{document}
\maketitle

\begin{abstract}
An edge-colored graph is \emph{rainbow} if all of its edges receive distinct colors.  
For a fixed graph $H$,  an   edge-colored graph $F$ is  called weakly  $H$-rainbow  saturated if there exists  an ordering $e_1,e_2,\ldots,e_{|E(\bar{F})|}$ of $E(\bar{F})$ such that, for any edge coloring $c$ of $E(\bar{F})$ with $c(e_i)\neq c(e_j)$, there is always a rainbow copy of $H$ that contains $e_i$ in $F+\{e_1,e_2,\ldots,e_i\}$.  The \emph{weak rainbow saturation number} $\operatorname{rwsat}(n,H)$ is the minimum number of edges in a weakly $H$-rainbow saturated graph on $n$ vertices. 
Li, Ma, and Xie [JGT, 2025] showed  that  
$\lim_{n\to\infty} \frac{\operatorname{rwsat}(n,H)}{n}$
exists for every nonempty graph $H$. 

Paths and stars attain, respectively, the minimum and maximum ordinary weak saturation numbers among all trees of the same order. We determine their weak rainbow saturation numbers exactly. For all $\ell>30$, we show that  $$ \ell+1=\s(n,P_\ell)< \s(n,S_\ell)=\binom{\ell}{2}-1$$ where $P_\ell$ and  $S_\ell$ denote the path and star on $\ell$ vertices, respectively. Thus, their dependence on $\ell$ is linear for paths and quadratic for stars. We then focus on cycles.  Li, Ma, and Xie asked whether $\operatorname{rwsat}(n,C_\ell)$ has leading term $\frac32n$ for every $\ell\ge4$. We answer this question negatively by giving an explicit construction showing that, for every $\ell\ge4$ and all sufficiently large $n$, $$\s(n,C_\ell)< \frac{\ell}{\ell-1}n+c_\ell,$$ 
where $c_\ell$ depends only on $\ell$. Since $\frac{\ell}{\ell-1}<\frac32$, this strictly improves the proposed leading coefficient for every cycle $C_\ell$ with $\ell\ge4$.
\end{abstract}
\noindent\textbf{Keywords:}  
saturation number; weak rainbow saturation number; rainbow coloring  
\section{Introduction}

Given a graph $F$, denote by $\Delta(F)$, $\delta(F)$, $e(F)$, $ n(F)$, and $G_F$ the \emph{maximum degree}, \emph{minimum degree}, number of edges, number of vertices of $F$, and the subgraph of $F$ that consists of all vertices of degree at least 1, respectively. Especially, call $n(F)$ the \emph{order} of $F$.  For a vertex $v$, denote by $N_F(v)$ and $d_F(v)$ the \emph{neighborhood} and \emph{degree} of $v$, respectively.  Especially, for a vertex $v$ of degree 1, denote by $e_v$ the unique edge incident with $v$ in $F$. 
For  $X\subseteq E(\bar{F})$, denote by $F+X$ the graph obtained from $F$ by adding edges in $X$. For $X\subseteq V(F)$, denote by  $\partial_F(X)$  the set of edges with one endpoint in
$X$ and the other end in $V(F)\backslash X$. In particular, $\partial_F(v)=\partial_F(\{v\})$ and $|\partial_F(v)|=d_F(v)$. For $i=0,1,\ldots,\Delta(F)$, let $R_i=\{v\in V(F):d_F(v)=i\}$. For $X\subseteq E(F)$ or $X\subseteq V(F)$, denote by $F-X$ the graph obtained from $F$ by deleting elements in $X$ and $F[X]$ the subgraph of $F$ induced by $X$.  We call $F$ a \emph{rainbow edge-colored} graph if all the edges of $F$ are colored differently. Throughout, we denote by   $c^*$ a rainbow edge coloring of $F$. For an integer $\ell$, denote by $P_\ell$, $S_\ell$, $C_\ell$, and  $K_\ell$ the path, star, cycle and complete graph on $\ell$ vertices, respectively.

Graph saturation is a classical topic in extremal graph theory. Given a graph $H$, a graph $F$ is called \emph{$H$-saturated} if $F$ contains no copy of 
$H$ as a subgraph, but for any $ e \in E(\bar{F})$, $F+  e $ contains a copy of 
$H$ as a subgraph, where $\bar{F}$ is the complement of $F$. The \emph{saturation number} $sat(n,H)$ is defined as the minimum number of edges of an $H$-saturated graph with $n$ vertices.  
  A graph $F$ on $n$ vertices is called \emph{weakly $H$-saturated} if there exists an ordering $ e_1 , e_2 ,\dots, e_t $ of the edges of $\overline{F}$ such that, setting $F_0=F$ and $F_i=F_{i-1}+ e_i $ for $1\le i\le t$, we have $F_t=K_n$, and for each $i$, the edge $ e_i $ belongs to a copy of $H$ in $F_i$ whose other edges are contained in $F_{i-1}$. The \emph{weak saturation number} of $H$, denoted by $\operatorname{wsat}(n,H)$, is the minimum number of edges in a weakly $H$-saturated graph on $n$ vertices.  Clearly,
$\operatorname{wsat}(n,H)\le \operatorname{sat}(n,H).$  For background on saturation and weak saturation, see
\cite{Bollobas1967,Erdos1964,Kaszonyi1986,Zykov1949}.
  
  


Saturation problems for edge-colored graphs were first studied by Hanson and Toft \cite{Hanson1987}.  In this paper, we are interested in edge-colored version weak saturation number. A related definition is rainbow saturation. 
Let $F$ be an edge-colored graph. We say that $F$ is \emph{rainbow $H$-saturated} if $F$ does not contain a rainbow copy of $H$ as a subgraph, but for any  edge $ e \in E(\bar{F})$ and any color assigned to $ {e}$, the resulting edge-colored graph $F+ {e}$ contains a rainbow copy of $H$. The \emph{rainbow saturation number}, denoted by $\operatorname{rsat}(n,H)$, is the minimum number of edges in a rainbow $H$-saturated graph on $n$ vertices. Gir\~ao, Lewis, and Popielarz \cite{Girao2020}  initiated the study of this parameter, and Behague,  Johnston, Letzter, Morrison, and  Ogden \cite{Behague2024} proved that $\operatorname{rsat}(n,H)=O_H(n)$ 
for every nonempty graph $ H $.


Behague et al. \cite{Behague2024} also introduced weak rainbow
saturation. For a fixed graph $H$, we say that an edge-colored graph $F$ is \emph{weakly  $H$-rainbow  saturated} if there exists an ordering $e_1,e_2,\ldots,e_{e(\bar{F})}$ of $E(\bar{F})$ such that, for any edge coloring $c$ of $E(\bar{F})$ with $c(e_i)\neq c(e_j)$, there is always a rainbow copy of $H$ that contains $e_i$ in  $F+\{e_1,e_2,\ldots,e_i\}$.   The corresponding edge ordering $e_1,e_2,\ldots,e_{e(\bar{F})}$  is called a \emph{weakly $H$-rainbow saturated ordering}. The \emph{weak rainbow saturation number} of $H $, denoted by $\operatorname{rwsat}(n,H)$, is the minimum number of edges in a weakly $H$-rainbow saturated graph on $n$ vertices. Denote by $\rw(n, H)$ the set of weakly $H$-rainbow saturated graphs on $n$ vertices with   $\operatorname{rwsat}(n,H)$ many edges. 

By the definitions, we have $\operatorname{wsat}(n, H)\leq \operatorname{rwsat}(n, H)\leq \operatorname{rsat}(n, H)$.  Behague et al. \cite{Behague2024} asked whether the limit
 $\lim_{n\to\infty}
\frac{\operatorname{rwsat}(n,H)}{n}$
 exists for every nonempty graph  $H $.
  Li, Ma, and Xie \cite{Li2025} proved a more precise result implying that $\lim_{n\to\infty} \frac{\operatorname{rwsat}(n,H)}{n}$
exists for every nonempty graph $H$. 

To state their result, for any graph $H$, let $f(H)$ be  the smallest integer $n$ such that for each $N\in \{n-1, n\}$ we have $\ex(N, \mathscr{H})\leq {N\choose 2}-2N-2$, where $\mathscr{H}= \{H-\{u, v\}\colon\, uv\in E(H)\}$.  Here   $\ex(n, \mathscr{F})$ is  the \emph{Tur\'{a}n number} of $\mathscr{F}$, that is, the maximum number of edges in an $n$-vertex graph that contains no member  of $\mathscr{F}$.  Their result may be stated as follows. 

\begin{theorem}[\cite{Li2025}]\label{th:rwsat_bounds}
Let $H$ be a non-empty graph. Then the following statements hold.
\begin{itemize}
\item[{\rm (i)}] If $H$ contains a vertex of degree 1 and $n>f(H)+1$, then $\s(n, H)\leq {f(H)+1 \choose 2}$.
\item[{\rm (ii)}] If $H$ contains no vertex of degree 1 and $n>f(H)+\delta(G_H)$, then
$$\frac{1}{2}\delta(G_H)n \leq\s(n, H)\leq \delta(G_H)(n-f(H)- \delta(G_H))+{f(H)+ \delta(G_H) \choose 2}.$$
\end{itemize}
\end{theorem}

Motivated by their result, we  investigate the exact weak rainbow saturation number for several fundamental graph classes. Saturation and rainbow variants have already been studied for several families
of trees and cycles; see, for example, \cite{CHEN2026,FFGJ2009,lane2024,LANE2026,xu2025s}.    Of particular relevance, paths and stars attain, respectively, the minimum and maximum ordinary weak saturation numbers among trees of a given order \cite{borowiecki2002weakly,faudree2013weak}. They are also extremal with respect to the number of leaves: an $\ell$-vertex path has two leaves, the minimum possible, whereas an $\ell$-vertex star has $\ell-1$ leaves, the maximum possible. We therefore begin by determining their weak rainbow saturation numbers exactly.

Our first result determines the weak rainbow saturation number of
paths.
\begin{theorem}\label{main:path}
For $n\ge \ell+1\ge 31$, it holds that $\s(n,P_\ell)= \ell+1$.     
\end{theorem}

Our second result determines the weak rainbow saturation number of
stars and shows that the corresponding extremal graphs are not
unique.

\begin{theorem}\label{main:star}
    Let $n$ and $\ell$ be integers with $n\ge 3\ell^2$ and  $\ell\ge 6$. Then $\s(n,S_{\ell})=\binom{\ell}{2}-1$. Furthermore, the extremal graphs are not unique. 
\end{theorem}

We next turn to cycles.  Li, Ma, and Xie \cite{Li2025} raised the following question about cycles.

\begin{question}[\cite{Li2025}]
For any integer $\ell \ge 4$, does there exist a constant $c_\ell$ such that 
\[
\operatorname{rwsat}(n, C_\ell) = \frac{3}{2}n + c_\ell?
\]
\end{question}
We answer this question  negatively for every $\ell\ge 4$. 
More precisely, by giving an explicit construction, we obtain an 
upper bound whose leading coefficient is strictly smaller than $\frac32$ for large enough $n$.

\begin{theorem}\label{main:cycle}
 For $\ell\ge 4$ and $n\ge 10\ell$, we have that 
    $$
   n+\frac{n}{6\ell}\le \s(n,C_\ell)\le \frac{n-4\ell}{\ell-1}  \ell +\binom{6\ell}{2}
    $$
\end{theorem}

We emphasize here that it was shown in \cite{Li2025} (Lemma 2.4) that any weakly $H$-rainbow saturated graph can be recolored to a rainbow one without losing the weak saturation property. Therefore, throughout this paper we may assume, without loss of generality, that the edge-coloring of 
$F$ is rainbow. Moreover, we denote by $c$ an injective mapping from $E(\bar{F})$ to the set of all  possible colors satisfying that  $c(e)\neq c(e')$ for distinct $e,e'\in E(\bar{F})$. Let $X\subseteq E(\bar{F})$, by an \emph{explicit assignment} of $c$ on $X$, we mean that  $c(e)$ is fixed for each $e$ in $X$.

The paper is organized as follows. In \Cref{sec:path}, \Cref{sec:star} and \Cref{sec:cycle}, we provide the proofs of \Cref{main:path}, \Cref{main:star} and \Cref{main:cycle}, respectively.

\section{Paths}\label{sec:path}
In this section, we focus on paths and prove that \Cref{main:path} holds.  

\begin{THMMAIN} 
For $n\ge \ell+1\ge 31$, it holds that $\s(n,P_\ell)= \ell+1$.     
\end{THMMAIN}
 The proof consists of the following parts. In \Cref{lem:ulbdd}, we show that  $\ell-1\le\s(n,P_\ell)\le \ell+1$ where $\ell\ge 30$.  We emphasize that the constant $30$ is not intended 
to be optimal; it is merely a convenient absolute threshold arising 
from our argument and could likely be reduced by refining the estimates.   Then  it remains to exclude $\ell-1$ and $\ell$. To exclude $\ell$ and $\ell-1$, suppose that there is a weakly $P_\ell$-rainbow saturated graph $F$ on $n$ vertices and $\ell-t$ edges where $t\in \{0,1\}$. Then the subgraph $ G_F$ of $F$ consists of non-isolated vertices are important. Briefly, taking the value of $t$  and the connectedness of $ G_F$ into consideration, there are four possibilities  for $ G_F$ in total. 
 
 We first introduce some tools in \Cref{prop:proA}, \Cref{lemma:edgenum}, and \Cref{lemma:strseq}. After that, we 
 consider first the case that $ e(G_F)=\ell-1$ in \Cref{lemma:l1e} with $ G_F$ connected and \Cref{lem:ct=1} with $ G_F$ disconnected.  Before addressing  the remaining  two cases that $t=0$ and $ G_F$ is connected or not, we first exclude three exact structures of $ G_F$ 
in the \Cref{lemma:subk4+}, \Cref{lemma:subedge} and  \Cref{lemma:substar}. Then we exclude the remaining cases that $ e(G_F)=\ell$ and $ G_F$ is connected in \Cref{lemma:conn2} while $ G_F$ is disconnected in \Cref{lem:disct=1}. 
\medskip

Now, we present a construction establishing the upper bound and an observation yielding the lower bound.

\begin{lemma}
    \label{lem:ulbdd}
Let $\ell\ge 30$ and $n\geq\ell+1$. Then $\ell-1\le\s(n,P_\ell)\le \ell+1$. Furthermore, $C_{\ell+1}\cup(n-(\ell+1))K_1$ is weakly $P_\ell$-rainbow saturated. 
\end{lemma}
\begin{proof}
To see the lower bound,  let $F\in \rw(n,P_\ell)$ and $e\in E(\bar{F})$ be the first edge added to $ F$. Then $F+\{e\}$ has at least two  rainbow $P_\ell$ containing $e$, otherwise, when $c(e)$ takes the color of the only  rainbow $P_\ell$ in $F+ e $, there is no  rainbow $P_\ell $ containing $e$ anymore in $F+ e $. It follows that $e(F)\ge \ell-1$.

To see the upper bound,  we show that $F=C_{\ell+1}\cup(n-(\ell+1))K_1$ is weakly $P_\ell$-rainbow saturated. Write $C_{\ell+1}=v_1v_2\cdots v_{\ell+1}v_1$, where the subscripts are considered modulo $\ell+1$.  
We add edges in $\bar{F}$ to $F$  according to  the following order.
\begin{itemize}
    \item [Step 1.] For each $i\in [\ell+1]$, add $v_iv_{i+2}$ to $ C_{\ell+1} $. \\
    Since $v_iv_{i+2}$ is contained in $C=v_1v_2\cdots v_iv_{i+2}\cdots v_{\ell+1}v_1$ which is of order $\ell$, we always obtain a rainbow $P_\ell$ containing $v_iv_{i+2}$ in $F+v_iv_{i+2}$.
    \item [Step 2.] For  each $i\in [\ell+1]$, add $v_iv_{i+3}$ to  $C_{\ell+1} $. \\
    We only use the newly added edges in Step 1 to obtain a rainbow $P_\ell$, and by definition, it must be rainbow-colored. If $\ell$ is odd, by symmetry, say $i=1$. Then  $$ v_{\ell+1} v_{\ell-1}v_{\ell-3}\cdots v_8v_6v_4v_1v_3v_5v_7\cdots v_{\ell-4}v_{\ell-2}v_{\ell}$$
    is a rainbow   $P_\ell$ which contains $v_1v_4$ while does not contain $v_2$.  
    If $\ell$ is even,  by symmetry, say $i=1$. Then  $$v_3v_5v_7\cdots v_{\ell-1}v_{\ell+1}v_2v_4v_1v_{\ell}v_{\ell-2}\cdots v_{10}v_{8}$$ 
    is a $P_\ell$ which contains $v_1v_4$ while does not contain $v_6$.   
    
    \item [Step 3.] For  each $i\in [\ell+1]$ and $4\le j \le \lfloor \frac{\ell+1}{2} \rfloor $, add $v_iv_{i+j}$ to $C_{\ell+1}$. \\

    Again, by symmetry, assume that $i=1$.     
   Note that if $\ell$ is odd, then after Step 1 and Step 2, $$C'=v_1v_3v_5\cdots v_{\ell-4}v_{\ell-2}v_{\ell}v_2v_4v_6\cdots v_{\ell-5}v_{\ell-3}v_{\ell-1}v_1$$  is a rainbow $C_\ell$ that does not contain $v_{\ell+1}$ and its edges are all newly added. 
   If $\ell$ is even, then after Step 1 and Step 2, $$C'=v_1v_3v_5\cdots v_{\ell- 3}v_{\ell-1}v_{\ell+1}v_2v_4v_6\cdots v_{\ell-4}v_{\ell-2}v_{\ell}v_1$$ is a rainbow  $C_{\ell+1}$  whose  edges are all newly added. 
 The graph $C'+v_1v_j$ contains a rainbow $P_\ell$ containing $v_1v_j$, since all the edges are newly added, $v_1,v_j\in V(C')$, and $|V(C')|\ge \ell$. 

   \item [Step 4.] For  each $i\in [\ell+1]$ and $w\in V(F)\setminus V(C_{\ell+1})$, add $v_iw$ to $F$.
 
  The existence of the cycle $C'$ guarantees a rainbow $P_\ell$ containing $v_iw$. 

 \item [Step 5.] For  each distinct $z,w\in V(F)\setminus V(C_{\ell+1})$, add $zw$ to $F$.

 By Step~4, there is a path of order  $\ell-1$ with one endpoint being $z$ 
   whose edges are all newly added in ‌previous steps.  By adding $zw$ to this path we obtain a rainbow $P_\ell$ containing $zw$.
   
\end{itemize}
Since all the edges have been added in order, the conclusion holds. 
\end{proof}



We say an edge $e\in E(\bar{F})$ is \emph{privileged} if we color $e$ with an arbitrary color $c(e)$, there is always a rainbow $P_\ell$ in $F+ e $ that contains $e$. For example, if $F$ is the path $v_1v_2\cdots v_\ell$, then the edge $v_1v_\ell$ is privileged. The following  proposition about privileged edges holds.  Note that by replacing $P_\ell$ with any graph $H$,   \Cref{prop:proA} still holds. 
\begin{proposition}\label{prop:proA}
    Given an  ordering  $\sigma =e_1,e_2,...,e_{e(\bar{F})}$ of
edges in $ E(\bar{F}) $. Let
$\sigma'=e'_1,\ldots,e'_{e(\bar{F})}$ be the ordering obtained by moving all privileged edges to the front while preserving their relative ordering, and
preserving the relative ordering of the remaining edges. Then $\sigma$ is weakly $P_\ell$-rainbow saturated  only if $\sigma'$ is   weakly $P_\ell$-rainbow saturated.
\end{proposition}

\begin{proof}
    Suppose that $\sigma$ is  weakly $P_\ell$-rainbow saturated   while $\sigma'$ is not. Then there is an integer $i$ and an edge coloring $c$ such that there is no rainbow $P_\ell$ in $F_i^{\sigma'}$ containing $e'_i$.  
    By the definition, $e_i'$ is not privileged. Then there is an integer $j\leq i$ such that $e_j=e'_i$ and $\{e_1,\ldots,e_j\}\subseteq \{e_1',\ldots,e_i'\}$. If we restrict the coloring $c$ to $F_j^\sigma$, there will be no rainbow $P_\ell$ in $F_j^\sigma$ containing $e_j$, otherwise this path will also be a rainbow $P_\ell$ in $F_i^{\sigma'}$ containing $e'_i$. This contradicts the assumption that $\sigma$ is  weakly $P_\ell$-rainbow saturated.
\end{proof}

By the above discussion, we may henceforth assume, without loss of generality, that all privileged edges appear first in the ordering.

Now, we introduce the following notions.    
Suppose that $\sigma=e_1,e_2,...,e_{e(\bar{F})}$ is an ordering of edges in  $E(\bar{F})$. Let $F_i^{\sigma}=F+ \{e_1,\ldots,e_i\}$. We omit the superscript $\sigma$ when it is clear from the context.  For $S\subseteq E(F_i)$ and an explicit assignment of $c$ on $\{e_1,\ldots,e_i\}$, define $$r _c^i(S):=\max_{P}\{|E(P)\cap S|: 
 P  \text{ is a rainbow $P_\ell$  in
$ F_i^{\sigma}$ under $c^*\cup c$ that contains $e_i$}\}.$$
Note that if $r_c^i(E(F_i))\leq \ell-2$,
then there is no rainbow copy of $P_\ell$ containing $e_i$ in
$ F_i^{\sigma}$ under the coloring $c\cup c^*$. Consequently, the ordering given above
is not a weakly $P_\ell$-rainbow saturated ordering. For convenience,   $r^i(S)\leq m$ means that   $r_c^i(S)\leq m$ holds for every   assignment of  $c$.


The following result helps to check whether a given ordering of $E(\bar{F})$ is weakly saturated.

\begin{lemma}\label{lemma:edgenum} 
Let  $F$  be a graph on $n$ vertices and $\ell-t$ edges with a  weakly $P_\ell$-rainbow saturated ordering $e_1,e_2,...,e_{e(\bar{F})}$, where $t\in \{0,1\}$ and $n(F)\ge \ell+1\ge 31$. Then for each $i\le \ell-2$,  there is no   $S\subseteq E(F)$ and an explicit assignment of $c$ such that all of  the following hold:
\begin{enumerate}[label=(\arabic*)] 
    
\item\label{1-1} $|E(F)\setminus S|\ge i$, $| S|\ge 2-t$, 
\item\label{1-2} $c(\{e_1,\ldots,e_i\})\subseteq c^*(E(F)\setminus S)$, and
 \item\label{1-3} $r_c^{i}(S)\leq |S|-2+t$.
\end{enumerate} 
\end{lemma}

\begin{proof} 
Suppose that for some $i$, the conclusion fails, i.e., there is $S\subseteq E(F)$ and an assignment of $c$ such that \ref{1-1}--\ref{1-3} hold. Since $ |E(F)\setminus S|\geq i$, there are pairwise distinct $f_1,\ldots,f_i\in E(F)\setminus S$. It suffices to assume that
$c(e_j)=c^*(f_j)$ for each $j\in [i]$.

Let $P$ be a rainbow $P_\ell$ in $F_i$ under $c^*\cup c$ that contains $e_i$. Since $c(e_j)=c^*(f_j)$ for each $j\in [i]$, for any $1\leq j\leq i$, at most one of $e_j$ and $f_j$ is contained in $P$. Assume that $e_i\in E(P)$ with $ E(P)\cap \{e_1,\ldots, e_i\}=\{e_{i_1},\ldots, e_{i_j}\}$ for some $1\le i_j\le i$. Then 
$$ E(P)\cap(E(F_i)\setminus S)\subseteq (E(F)\setminus S)\cup \{e_{i_1},\ldots,  e_{i_j}\},$$
where $\{f_{i_1},\ldots, f_{i_j}\}\subseteq E(F)\setminus S$. 
Consequently, $$|E(P)\cap(E(F_i)\setminus S)|
\leq e(F)-|S|=\ell-t-|S|.$$ 
Since  $|E(P)\cap S|\leq r_c^i(S)$, it holds that  $$|E(P)|=|E(P)\cap S|+|E(P)\cap(E(F_i)\setminus S)|\leq r_c^i(S) +\ell-t-|S| \leq \ell-2.$$
Hence,  there is no rainbow $P_\ell$ in $F_i$ that contains $e_i$. It follows that  the  given  ordering 
is not weakly $P_\ell$-rainbow saturated, a contradiction. 
\end{proof}

Let $P_4^+$ be the graph obtained from  $P_4=v_1v_2v_3v_4$ 
by adding the edge $v_2v_4$. 
Note that this is the unique simple graph with
degree sequence $(1,2,2,3)$. Moreover, the following result holds. 

\begin{lemma}\label{lemma:strseq}
 Let $H$ be a connected graph with $\ell$ vertices and degree sequence being $(1,2,2,\ldots,2,3)$. Then $H$ is obtained from the path
$v_1v_2\cdots v_{\ell}$  by adding $v_iv_{\ell}$  for some $2\leq i\leq \ell-2$.
\end{lemma}

\begin{proof}
    We prove it by induction on the number of vertices. If $ n(H)=4$, then $H$ is isomorphic to $P_4^+$. Now assume that $ n(H)\geq 5$, $d_H(v_1)=1$ and $v_1v_2\in E(H)$. Let $H'$ be the graph obtained from $H$ by deleting $v_1$.  If $d_H(v_2)=2$, then the degree sequence  of $H'$ is $(1,2,2,\ldots,3)$. By the induction hypothesis, $H'$ is obtained from the path $v_2v_3\cdots v_\ell$ by adding $v_iv_\ell$ for some $3\leq i\leq\ell-2$. Adding back $v_1$, we know that $H$ is obtained from the path $v_1v_2v_3\cdots v_\ell$ by adding $v_iv_\ell$ for some $3\leq i\leq\ell-2$.  If $d_H(v_2)=3$, then $H'$ is connected with degree sequence  $(2,2,\ldots,2)$. Then $H'$ is a cycle $c_2c_3\cdots c_\ell c_2$. Consequently,  $H$ is obtained from the path $v_1v_2v_3\cdots v_\ell$ by adding $v_2v_\ell$. 
\end{proof}

 Now, we consider  $ e(G_F)=\ell-1$ in \Cref{lemma:l1e} with $ G_F$ connected and \Cref{lem:ct=1} with $ G_F$ disconnected. 

\begin{lemma}\label{lemma:l1e} 
Let $F$ be a graph on $n$ vertices, where $n\ge \ell+1\ge 31$. If $ G_F$ is a connected graph with $\ell-1$ edges, then $F\notin  \rw(n,P_\ell)$.
\end{lemma}	
\begin{proof}
Suppose $F\in  \rw(n,P_\ell)$. Then let $e_1,e_2,...,e_{e(\bar{F})}$ be weakly $P_\ell$-rainbow saturated. Note that necessarily    $  n(G_F)\geq \ell-1$.

If $  n(G_F)=\ell-1$, then $e_1=zw$ must be an edge connecting $ G_F$ and $F\backslash  G_F$. 
Assume that $z\in V(G_F)$ and $w\in V(F)\setminus V(G_F)$. If $d_{G_F}(z)=1$,  when  
	$c(e_1)=c^*(e_z)$, it holds that $r^1_c(E(G_F))=0$. So we assume that
$2\leq d_{G_F}(z)\leq \ell-2$. Let $S=\partial_{G_F}(z)$, then $r^1(S)\leq 1\leq |S|-1$ and $|E(F)\setminus S|\geq 1$.
By Lemma~\ref{lemma:edgenum}, the ordering is not  
weakly $P_\ell$-rainbow saturated, a contradiction.

Now, suppose that $  n(G_F)=\ell$. Then $ G_F$ is a tree. 
If $|R_1|\geq 3$, then there is a vertex $x$
with $d(x)\geq 3$.  If $d(x)=\ell-1$, then $ G_F$ is the star $S_\ell$ and  $r^1(E(F))\leq 2$. Thus, assume that $d(x)\leq\ell-2$. 
Set $S=\partial_{G_F}(x)$, then $r^1(S)\leq 2\leq |S|-1$ and $|E(F)\setminus S|\geq 1$.
By Lemma~\ref{lemma:edgenum}, the ordering is not  weakly $P_\ell$-rainbow saturated, a contradiction. 
Hence,  $ G_F$ is a path with $R_1=\{u,v\}$.  Then $e_1=uv$ is privileged. Let $e_2=zw$. If  $z\in V(G_F)\backslash R_1$,  by setting $S=\partial_{G_F}(z)$,  we have
$r^2(S)\leq 1\leq |S|-1$ and $|E(F)\setminus S|\geq 2$.
By Lemma~\ref{lemma:edgenum}, the ordering is not  
weakly $P_\ell$-rainbow saturated, a contradiction. Thus, without loss of generality, assume that  $w\in V(F)\setminus V(G_F)$
and $z=u$.   Then when setting $c(e_2)=c^*(e_u)$ and
$c(e_1)=c^*(e_v)$,  the longest rainbow path in $F_2$ containing $e_2$ is $wzv$, a contradiction.  

Therefore, $F\notin \rw(n,P_\ell)$.
\end{proof}

\begin{lemma}\label{lem:ct=1}
Let $F$ be a graph on $n$ vertices, where $n\ge \ell+1\ge 31$. If $ G_F$ is  disconnected with $\ell-1$ edges, then $F\notin  \rw(n,P_\ell)$. 
\end{lemma}

\begin{proof}
As before, suppose that $F\in  \rw(n,P_\ell)$.
Since every component of $ G_F$ has at most $\ell-2$ edges, $e_1=zw$ must be an edge 
connecting two components of $ G_F$. Suppose not, that $z,w\in V(G_1)$ or $z\in V(G_1)$ and $w\in V(F)\setminus V(G_F)$. Then when $c(zw)\in c^*(E(G_1))$, there is no rainbow $P_\ell$ containing $zw$ in $F_1$. Thus, let $G_1$ and $G_2$ be two components of $ G_F$
such that $z\in V(G_1)$ and $w\in V(G_2)$ and $ n(G_1)\ge  n(G_2)$.  

 If $d_{G_F}(z)\geq 2$, by setting  $S=\partial_{G_F}(z)$,   we have $2\leq|S|\leq\ell-2$ and $r^1(S)\leq1$. Lemma~\ref{lemma:edgenum}~ gives a contradiction. The case when $d_{G_F}(w)\geq 2$ is same. 
Thus, assume that $d_{G_F}(z)=d_{G_F}(w)=1$. Then let $c(e_1)=c^*(e_z)$. Consequently, the longest rainbow path in $F+ \{e_1\}$ containing $e_1$ has order  at most $1+ n(G_2)\leq\ell-2$ since $ n(G_1)\geq n(G_2)$, a  contradiction. 
\end{proof}

As addressed before, we first exclude three exact structures of $ G_F$:
\begin{itemize}
    \item $ G_F$ is obtained from the path
$v_1v_2\cdots v_{\ell}$  by adding $v_iv_{\ell}$  for some $2\leq i\leq \ell-2$;
\item $ G_F$ is the path $v_1v_2\cdots v_{\ell+1}$;
\item  $ G_F$  is the union of three internally vertex-disjoint paths $Q_1,\; Q_2,\; Q_3,$ sharing a unique endpoint $x$, with $ e(G_F)=\ell$, 
\end{itemize}
in \Cref{lemma:subk4+}, \Cref{lemma:subedge} and  \Cref{lemma:substar}, respectively.

\begin{lemma}\label{lemma:subk4+}
Let $F$ be a graph on $n$ vertices, where $n\ge \ell+1\ge 31$. If $ G_F$ is the graph obtained from the path
$v_1v_2\cdots v_{\ell}$  by adding $v_iv_{\ell}$  for some $2\leq i\leq \ell-2$,  then $F\notin  \rw(n,P_\ell)$.
\end{lemma}

\begin{proof}
Suppose  that $F\in  \rw(n,P_\ell)$ with a weakly $P_\ell$-rainbow saturated ordering $e_1,e_2,...,e_{e(\bar{F})}$ of
edges in $E(\bar{F})$.   By Proposition \ref{prop:proA},   we may assume that $e_1=v_1v_{\ell}$ and $e_2=v_1v_{i+1}$. Write $e_3=zw$.
Let $Q$ be the path $v_1v_2\cdots v_{i-1}$ and $C$
be the cycle $v_iv_{i+1}\cdots v_{\ell}v_{i}$.

We first eliminate several possibilities for the locations of $z$ and $w$ by symmetry of $z$ and $w$.
\begin{itemize}
\item If $z\in V(G_F)\setminus (R_1\cup N[v_i])$, let $
S=\partial_{G_F}(z)\cup\partial_{G_F}(v_i)
$.
Then $|S|=5$ and $r^3(S)\leq 3$.

\item If $z,w\in N(v_i)$, let
$S=\partial_{G_F}(z)\cup \partial_{G_F}(w)$. Note that if $i=\ell-2$, then $zw\neq v_{\ell-1}v_\ell$. Moreover, if $i=2$, then $zw\neq e_2$ since $e_2=v_1v_3$.  
Then $|S|=4$ and $r^3(S)\le2$. 

\item If $z=v_i$, let $S=\partial_{G_F}(v_{i})$. Then $|S|=3$ and $r^3(S)\le1$.
\end{itemize}
In each case, Lemma~\ref{lemma:edgenum} implies that the given edge ordering cannot be weakly $P_\ell$-rainbow saturated. Hence,  none of the above cases can occur.  We now distinguish the remaining possibilities.

\medskip
\noindent\textbf{Case 1.} $e_3\in E(\bar {G}_F)$.

By the preceding discussion, it follows that $e_3=v_1v_{i-1}$.  
Suppose first that $i\le\ell-4$. Define $S=S_1\cup S_2$ where $S_1=E(Q)$ and $S_2=\{v_{i-1}v_i,v_iv_{i+1},v_iv_{\ell}\}$. 
We have $r^3(S_1)\leq |E(Q)|-1= i-3$ and  $r^3(S_2)\leq 2$. Thus, $r^3(S)\leq i-1$. Since $|S|=i+1\leq \ell-3$, Lemma~\ref{lemma:edgenum} yields a contradiction. Now assume that $\ell-3\leq i\leq \ell-2$. Define $S=S_1\cup S_2$ where $S_1=\{v_{i-2}v_{i-1},v_{i-1}v_{i} \}$ and $S_2=E(C)$. Then $r^3(S_1)\leq 1$ and $r^3(S_2)\leq \ell-i$. Thus, $r^3(S)\leq \ell+1-i$. Since $|S|=\ell-i+3\leq 6\leq\ell-3$, again Lemma~\ref{lemma:edgenum} gives a contradiction.

Henceforth, assume that $e_3\in E(\bar{F})\setminus E(\bar{G}_F)$. Without loss of generality, suppose that $z\in V(G_F)$ and $w\in V(F)\setminus V(G_F)$. By the initial discussion, $z\in R_1\cup N(v_{i})$. By symmetry of $v_{i+1}$ and $v_\ell$, it is sufficient to consider that $z\in \{v_1,v_{i-1},v_{i+1}\}$.  

\medskip
\noindent\textbf{Case 2.} $z=v_1$.

Let $c(e_3)=c^*(v_{i-1}v_i)$. Suppose first that $i\ge4$, and let $P$ be a rainbow path containing $e_3$. If $V(P)\cap V(Q)\neq \{v_1\}$, then  $V(P)\cap V(C)= \emptyset$. Conversely, if $V(P)\cap V(C)\neq \emptyset$, then  $V(P)\cap V(Q)= \{v_1\}$. Since $ n(Q)\ge3$ and $ n(C)\ge 3$, it follows that $ n(P)\leq \ell-1$.

Now, assume that $2\leq i\leq 3$. Let $c(e_1)=c^*(v_{5}v_6)$ and $c(e_2)=c^*(v_{6}v_7)$. Again let $P$ be a rainbow path containing $e_3$. If neither $e_1$ nor $e_2$ belongs to $P$, then $ n(P)\leq 3$.
If $e_1\in E(P)$, then $v_5v_6\notin E(P)$, implying $ n(P)\leq\ell-1$.
Similarly, if $e_2\in E(P)$, then $v_6v_7\notin E(P)$, and again $ n(P)\leq\ell-1$.

\medskip
\noindent\textbf{Case 3.}  $z  =v_{i-1}$.

By Case 2, we may assume that $i\geq 3$ when $z=v_{i-1}$. 

 
Since $\ell\ge 30$, we may let $S=\partial_F( v_1 )\cup\partial_F( z) \cup\partial_F(v_{i})\cup\partial_F(v_{i+1} )\cup\partial_F(v_{\ell} )$ and set $c(\{e_1,e_2,e_3\})\subseteq c^*(E(F)\setminus S)$.  Let $P$ be  a rainbow path containing $e_3$ in $F_3$. If $P$ contains $v_1$, then $|E(P)\cap (\partial_F( z)\cup\partial_F(v_{i+1} )\cup\partial_F( v_\ell ))|\le |\partial_F( z)\cup\partial_F(v_{i+1} )\cup\partial_F( v_\ell )|-2$. If $P$ does not contain $v_1$, then  $|E(P)\cap (\partial_F( v_1)\cup\partial_F( v_{i}))|\le |\partial_F( v_1)\cup\partial_F( v_{i})|-2$. Since $S=\partial_F( v_1 )\cup\partial_F( z) \cup\partial_F(v_{i})\cup\partial_F(v_{\ell} )$,  Lemma~\ref{lemma:edgenum} gives a contradiction.


\medskip
\noindent\textbf{Case 4.} $z=v_{i+1}$.

 In the following, we will pick three edges $\{f_1,f_2,f_3\}\subseteq E(F)$ depending on specific cases so that we can set $c(e_i)=c^*(f_i)$. Then let $R=\{ f_1,f_2,f_3\}$, $S=E(F)\setminus R$ and $P$ be a rainbow path containing $e_3$ in $F_3$.
 Consequently, by  \Cref{lemma:edgenum}, it suffices to show that $r_c^3(S)\le |S|-2$.

If $i\leq\ell-4$, set $c(e_3=wv_{i+1})=c^*(v_iv_{\ell})$, $c(e_2=v_1v_{i+1})=c^*(v_1v_2)$ and
$c(e_1=v_1v_{\ell})=c^*(v_{\ell-1}v_{\ell})$. By the setting of $c$ on $\{e_1,e_2,e_3\}$, to make $|E(P)\cap S|$ as large as possible, we know that $P$ does not contain $e_1$ or $e_2$. Then  $E(P)$ is disjoint with $\{v_2v_3,v_3v_4\}\subseteq S$ or $\{v_{\ell-3}v_{\ell-2},v_{\ell-2}v_{\ell-1}\}\subseteq S$, a contradiction.


If $i\in \{\ell-3,\ell-2\} $, set $c(e_3=wv_{i+1})=c^*(v_iv_{\ell})$, $c(e_2=v_1v_{i+1})=c^*(v_7v_8)$ and
$c(e_1=v_1v_{\ell})=c^*(v_{8 }v_{9})$.  By the same argument as above, we get a contradiction.

\medskip
Therefore, the conclusion holds. 
\end{proof}

\begin{lemma}\label{lemma:subedge} 
Let $F$ be a graph on $n$ vertices, where $n\ge \ell+1\ge 31$. If  $ G_F$ is the path $v_1v_2\cdots v_{\ell+1}$, then $F\notin  \rw(n,P_\ell)$. 
\end{lemma}

\begin{proof}
Suppose   that $F\in  \rw(n,P_\ell)$ and let $
e_1,e_2,\ldots,e_{e(\bar{F})} 
$ 
be a weakly $P_\ell$-rainbow saturated  ordering of $E(\bar F)$. By Proposition \ref{prop:proA},  we may assume that
$e_1=v_1v_{\ell+1}$, $e_2=v_1v_{\ell}$
and $e_3=v_2v_{\ell+1}$. Write $e_4=zw$.

If $z,w\in V(G_F)\setminus R_1$, define
$S=\partial_{G_F}(z)\cup \partial_{G_F}(w)$. We have $r^4(S)\leq 2$ and $|S|=4$. By Lemma~\ref{lemma:edgenum}, 
 the given ordering is not weakly $P_\ell$-rainbow saturated, a contradiction.

 Next, suppose that $e_4\in E(\bar{G}_F)$ with $z=v_1$ and $w\notin  \{v_\ell, v_{\ell+1}\}$ or  $e_4\in E(F)\setminus E(G_F)$ with $z=v_i\in  V(G_F)$. In the following, we will pick three edges $\{f_1,f_2,f_3,f_4\}\subseteq E(F)$ depending on specific cases so that we can set $c(e_i)=c^*(f_i)$. Then let $R=\{ f_1,f_2,f_3,f_4\}$, $S=E(F)\setminus R$.   
 
 If $i\ge 12$, then let  $c(zw)=c^*(v_{9}v_{10})$, $c(v_1v_{\ell+1})=c^*(v_4v_5)$, $c(v_1v_{\ell})=c^*(v_{5 }v_{6})$,  and $c(v_2v_{\ell+1})=c^*(v_{6}v_{7})$.   To make $|E(P)\cap S|$ as large as possible, if $P$ contains $e_i$  for some $i\in [3]$, then $E(P)$ is disjoint with  $\{v_7v_8,v_8v_9\}\subseteq S$, a contradiction.  Suppose that $E(P)$ is disjoint with $\{e_1,e_2,e_3\}$. Again, $E(P)$ is disjoint with $\{v_{1 }v_{ 2},v_{2}v_{3}\}\subseteq S$ or $\{v_{10}v_{11},v_{11}v_{12}\}\subseteq S$ or $\{v_{i}v_{i+1},v_{i+1}v_{i+2}\}\subseteq S$, a contradiction. Thus, assume that $3\le i\le 11$. Then by setting  $c(zw)=c^*(v_{13}v_{14})$, $c(v_1v_{\ell+1})=c^*(v_{16}v_{17})$, $c(v_1v_{\ell})=c^*(v_{17 }v_{18})$,  and $c(v_2v_{\ell+1})=c^*(v_{18}v_{19})$, the same argument works.

Therefore, $F\notin \rw(n,P_\ell)$. 
\end{proof}

\begin{lemma}\label{lemma:substar}
Let $F$ be a graph on $n$ vertices, where $n\ge \ell+1\ge 31$. If  $ G_F$ is the union of three internally vertex-disjoint  paths $Q_1,\; Q_2,\; Q_3,$  where they pairwise share the unique endpoint $x$, with $e(Q_i)\ge 1$ and $ e(G_F)=\ell$, then $F\notin  \rw(n,P_\ell)$. 
\end{lemma}

\begin{proof}
 Similarly, let $
e_1,e_2,\ldots,e_{e(\bar{F})}
$
be a  weakly $P_\ell$-rainbow saturated  ordering  of $E(\bar F)$.
 Without loss of generality, assume that  $ n(Q_3) \leq  n (Q_2) \leq  n(Q_1) $. For $i=1,2,3$, let $g_i$ be the edge of $Q_i$ that is incident with $x$.   
 
 First, we claim that $ G_F$ is obtained from a path $v_1v_2\cdots v_{\ell}$ by adding a pendant  edge $v_rv_r'$ where $2\leq r\leq\ell-1$. We prove this by showing that  $ n(Q_3)=2$. Assume to the contrary  that $ n(Q_3)\geq 3$. 
If $e_1\in E(\bar{G}_F)$, choose a path $Q_j$ from $Q_1,Q_2,Q_3$ that is internally disjoint from $e_1$. Let $c(e_1)=c^*(g_j)$. Then any rainbow path $P$ in $G_F+ e_1 $ containing $e_1$ contains no edge of $Q_j$.  Since $ n(G_F)=\ell+1$ and $ n(Q_i)\geq3$ for each $i\in[3]$, it holds that $ n(P)\leq \ell-1$, a contradiction.   Thus, suppose that  $e_1=zw$, where $w\in V(F)\setminus V(G_F)$ and $z\in V(G_F)$. If $d(z)=1$, by setting $c(zw)=c^*(e_z)$, the longest rainbow path containing $zw$ in $F_1$ is $zw$. If $z=x$, then the longest rainbow path containing $zw$ is $Q_1\cup\{zw\}$. If $d(z)=2$, then there exists a unique $Q_i\in\{Q_1,Q_2,Q_3\}$ such that $z\in V(Q_i)$. Let $c(zw)=c^*(g_i)$. Then the order  of the longest rainbow path containing $zw$ is at most $ n(Q_i)\leq \ell-3$. These are all contradicting the assumption that the given ordering is weakly $P_\ell$-rainbow saturated. Therefore, $ n(Q_3)=2$.  

By symmetry, it suffices to assume that $3\leq r\leq\ell-2$  or $r=\ell-1$.  We finish the proof by considering the two cases of $r$.

\noindent\textbf{Case 1.}  $3\leq r\leq  \frac{\ell+1}{2}$. 

Without loss of generality,   assume that $e_1=v_1v_{\ell}$. Write $e_2=zw$ with $z\in V(G_F)$.
\begin{itemize}
    \item If $z\notin \{v_1,v_{r-1},v_r,v_r', v_{r+1},v_\ell\}$, then let $S=\partial_{G_F}(z)\cup\partial_{G_F}(v_r)$. We have $|S|=5$ and $r^2(S)\leq3$.

    \item If $z=v_r$, then let $S=\partial_{G_F}(v_r)$. We have $|S|=3$ and $r^2(S)\leq 1$.    
     
    \item If   $z\in \{v_{r-1}, v_{r+1}\}$ and $w\neq v'_r$, then let $S=\partial_F(v_1)\cup \partial_F(v_{r-1})\cup\partial_F(v_r)\cup \partial_F(v_{r+1})\cup \partial_F(v_{\ell}) \cup \partial_F(w)$  and set $c(\{e_1,e_2\})\subseteq c^*(E(F)\setminus S)$. Let $P$ be a rainbow path in $F_2$ containing $e_2=zw$. If $P$ does not contain $v_1v_2$ or $v_{\ell-1}v_\ell$, then $|E(P)\cap (\partial_F(v_1)\cup \partial_F(v_r)\cup\partial_F(v_\ell) ) |\le |\partial_F(v_1)\cup \partial_F(v_r)\cup\partial_F(v_\ell)|-2$. Thus, assume that $P$ contains $v_1v_2$ and $v_{\ell-1}v_\ell$. If $w\notin V(G_F)$, since $v_1v_2,v_{\ell-1}v_\ell\in E(P)$ and $P$ is rainbow, we know that $|E(P)\cap(\partial_F(v_{r-1})\cup\partial_F(v_r)\cup \partial_F(v_{r+1}))|\le |\partial_F(v_{r-1})\cup\partial_F(v_r)\cup \partial_F(v_{r+1})|-2$.   Therefore, $w\in V(G_F)\setminus \{v'_r\}$. If $d(w)=2$, then $|E(P)\cap (\partial(w)\cup \partial(z))|\le |\partial(w)\cup \partial(z)|-2$. Thus, $w\in \{v_1,v_\ell\}$.  Then either $P$ does not contain $v_rv'_r$ or $ v_1v_2$ or $v_{\ell-1}v_\ell$, a contradiction. 
\end{itemize}
In each case, Lemma~\ref{lemma:edgenum} implies that the given edge ordering cannot be  weakly $P_\ell$-rainbow saturated. As before,  we will pick three edges $\{f_1,f_2 \}\subseteq E(F)$ depending on specific cases so that we can set $c(e_i)=c^*(f_i)$. Then let $R=\{ f_1,f_2 \}$, $S=E(F)\setminus R$ , $P$ be a rainbow path in $F_2$ containing $e_2$.  Since $v_1$ and $v_\ell$  are analogous, it remains to consider the following cases, where we always set $c(v_1v_\ell)=c^*(v_{\ell-1}v_{\ell})$.  

  \begin{itemize}
 \item   $z\in \{v_{r-1}, v_{r+1}\}$ and $w=v_r'$.  Let $c(e_2)=c^*(v_{\ell-6}v_{\ell-5})$. 
 \item   $z\in \{v_{1}, v_{\ell}\}$ and $w=v_r'$.  Let $c(e_2)=c^*(v_{\ell-6}v_{\ell-5})$. 
      \item  $w\in V(F)\setminus V(G_F)$ and $z=v_1$ or $z=v'_r$. Then set $c(zw)=c^*(v_{1}v_2)$ when $z=v_1$ and $c(zw)=c^*(v_{r}v'_{r})$  when $z=v'_{r}$. 
  \end{itemize}
In each case, Lemma~\ref{lemma:edgenum} implies that the given edge ordering cannot be   weakly $P_\ell$-rainbow saturated. 

      




\medskip
\noindent\textbf{Case 2.} $r=\ell-1$. 

Without loss of generality, we may assume that $e_1=v_1v_{\ell}$ and $e_2=v_1v_{\ell-1}'$. Write $e_3=zw$ with $z\in V(G_F)$. 
\begin{itemize}
    \item If $z\notin \{v_1,v_{\ell-2},v_{\ell-1},v_{\ell-1}',v_\ell\}$, then let $S=\partial_{G_F}(z)\cup\partial_{G_F}(v_{\ell-1})$. We have $|S|=5$ and $r^3(S)\leq3$.

    \item If $z=v_{\ell-1}$, then let $S=\partial_{G_F}(v_{\ell-1})$. We have $|S|=3$ and $r^3(S)\leq 1$.
     
\end{itemize}
  In each case, Lemma~\ref{lemma:edgenum} implies that the given edge ordering cannot be   weakly $P_\ell$-rainbow saturated.  By symmetry of $v'_{\ell-1}$ and $v_{\ell}$, it remains to consider the following cases.  Let $ R$, $S $ and $P$ be defined as before. 

    \begin{itemize}
       \item $z=v_1$. Then set  $c(e_1=v_1v_\ell)=c^*(v_{\ell-1} {v_\ell})$, $c(e_2=v_1v_{\ell-1}')=c^*(v_{\ell-1} v_{\ell-1}')$ and $c(v_1w)=c^*(v_{i}v_{i+1})$ for $w=v_i\in V(G_F)\setminus \{v_{\ell-1}\}$; $c(v_1w)=c^*(v_{7}v_8)$ for $w=  v_{\ell-1} $;  while $c(v_1w)=c^*(v_{1}v_{2} )$ for $w\notin V(G_F)$.   
        \item $z=v_{\ell-2}$ and $w\neq v_1$. Then set  $c(e_1=v_1v_\ell)=c^*(v_{\ell-1} {v_\ell})$, $c(e_2=v_1v_{\ell-1}')=c^*(v_{\ell-1} v_{\ell-1}')$ and $c(zw)=c^*(v_{k-1}v_{k})$ for $w=v_k\in V(G_F)\setminus \{v_1,v'_{\ell-1},v_\ell\}$; $c(zw)=c^*(v_{\ell-3}v_{\ell-2})$ for $w\in \{v'_{\ell-1},v_\ell\} \cup (V(F)\setminus  V(G_F))$.   

      \item  $w=v'_{\ell-1}$ and $z=v_{\ell}$. Then set $c(zw)=c^*(v_{4}v_{5})$, $c(v_1v_{\ell-1}')=c^*(v_2v_3)$ and $c(v_1v_\ell)=c^*(v_{6}v_7)$.

      \item  $w\in V(F)\setminus V(G_F)$ and $z=v_\ell$. Then set $c(zw)=c^*(v_{\ell-1}v_\ell)$, $c(v_1v_{\ell-1}')=c^*(v_{\ell-1}v_{\ell-1}')$ and $c(v_1v_\ell)=c^*(v_1v_2)$.
  \end{itemize}
 The inequality $r^3_c(S)\le |S|-2$ follows in each case. 
 Hence,  Lemma~\ref{lemma:edgenum} implies that the given edge ordering cannot be   weakly $P_\ell$-rainbow saturated, a contradiction.
\end{proof}

Now,  we finish the proof in  \Cref{lemma:conn2}  and \Cref{lem:disct=1}. 
\begin{lemma} \label{lemma:conn2}
Let $F$ be a graph on $n$ vertices, where $n\ge \ell+1\ge 31$. If $ G_F$ is a connected graph with $\ell$ edges, then $F\notin  \rw(n,P_\ell)$.
\end{lemma}

\begin{proof}
As before, suppose that  $e_1,\dots, e_{e(\bar{F})}$ is a weakly $P_\ell$-rainbow saturated ordering of $E(\bar{F})$.  

Firstly, we assert that  $\Delta(G_F)\leq 3$. Assume that there is a vertex $x$
with $d(x)\geq4$.   If $d(x)=\ell$, then $ G_F$ is the star $S_{\ell+1}$ and  $r^1(E(F))\leq 2$. Thus, assume that $d(x)\leq\ell-1$.  Set $S=\partial_{G_F}(x)$, then 
$r^1(S)\leq 2\leq |S|-2$ and $|E(F)\setminus S|\geq 1$.
By Lemma~\ref{lemma:edgenum}, the ordering is not 
weakly $P_\ell$-rainbow saturated, a contradiction.

Since $ e(G_F)=\ell$ and  $F$ is weakly $P_\ell$-rainbow saturated, it holds that $  n(G_F)\in \{\ell-1,\ell,\ell+1\}$. Then it suffices to consider the following cases. 

\medskip
\noindent\textbf{Case 1.} 
 $  n(G_F)= \ell-1$. 
 
Write $e_1=zw$. Since $  n(G_F)= \ell-1$, we assume that $z\in V(G_F)$ and $w\in V(F)\backslash V(G_F)$.  

If $R_1\neq \emptyset$, let $x\in R_1$. Now set $c(zw)=c^*(e_x)$. If $z=x$
then the longest rainbow path in $F_1$ containing $e_1$ is $zw$.  If $z\neq x$ then any rainbow path containing $zw$
has  order   at most $\ell-1$ since it cannot contain $x$. 
Thus, assume that $R_1=\emptyset$.  Consequently,  $|R_3|=2$, say $R_3=\{x,y\}$.  It suffices to consider the following cases.

\begin{itemize}
	\item $z\in R_3$. Set $S=\partial_{G_F}(z)$, then $|S|=3$ and $r^1(S)\leq 1$.
\item $d_{G_F}(z)=2$ and $z\notin N(x)$. Set $S=\partial_{G_F}(z)\cup \partial_{G_F}(x)$, then
	$|S|=5$ and $r^1(S)\leq 3$.
\item $x$ and $y$ are not adjacent. Set $S=\partial_{G_F}(x)\cup \partial_{G_F}(y)$, then
	$|S|=6$ and $r^1(S)\leq 4$.
\item $x$ and $y$ are adjacent and $z\in N(x)\cap N(y)$. Set $S=\partial_{G_F}(x)\cup \partial_{G_F}(y)$, then
	$|S|=5$ and $r^1(S)\leq 3$.
\end{itemize}
Among all the situations above,  we  have $r^1(S)\leq |S|-2$ and $|E(F)\setminus S|\geq 1$.
So by Lemma~\ref{lemma:edgenum}, the ordering is not  
weakly $P_\ell$-rainbow saturated, a contradiction.

\medskip
\noindent\textbf{Case 2.}  $  n(G_F)=\ell$. 

We proceed by considering $|R_1|$.

Suppose first that $|R_1|\geq 3$. If $e_1 \in E(\bar{G}_F)$, then there is an $x\in R_1$ 
not incident with $e_1$. Set $c(e_1)=c^*(e_x)$. Then any rainbow path containing $e_1$
has  order  at most $\ell-1$, since it cannot contain $x$. Thus, assume that
$e_1=zw$, where $z\in V(G_F)$ and $w\in V(F)\backslash V(G_F)$. If $z\in R_1$ then  
by setting $c(e_1)=c^*(e_z)$,  the longest rainbow path in $F_1$  is $wz$.  If $z\notin R_1$, let $S$  be the set of edges 
incident with $R_1$.  Then it is enough to assume that  $\ell-1\geq|S|=|R_1|\geq 3$ and so $r^1(S)=1$. However, 
$|E(F)\setminus S|\geq 1$. By Lemma~\ref{lemma:edgenum}, this leads to a contradiction.

Now, suppose that $|R_1|=2$. Since $\Delta(G_F)\leq 3$ and $e(G_F)=\ell$,  we have that $|R_3|=2$. Suppose that $R_1=\{u,v\}$ and 
$R_3=\{x,y\}$.  
If $x$ and $y$ are not adjacent, set $S=\partial_{G_F}(x)\cup\partial_{G_F}(y)$. 
We have $|S|=6$ and $r^1(S)\leq 4$. Again, by Lemma~\ref{lemma:edgenum}, we are done. 
Thus, assume that $x$ and $y$ are adjacent and  $e_1=uv$. Write $e_2=zw$ with $z\in V(G_F)$.
 
 Since $e_1=uv$ and $R_1=\{u,v\}$, there is a vertex incident with $e_2$ that is not in $R_1$. Without loss of generality, we can assume that $z\notin R_1$. Then it suffices to consider the following cases. 
\begin{itemize}
	\item $z\in R_3$. Let $S=\partial_{G_F}(z)$.
		Then $|S|=3$ and $r^2(S)\leq 1$.
	\item $z\notin R_3\cup N(x)$. Let $S=\partial_{G_F}(x)\cup \partial_{G_F}(z)$.
		Then $|S|=5$ and $r^2(S)\leq 3$.
	\item $z\in N(x)\cap N(y)$. Note that now $xyzx$ is a cycle of order $3$.  Set $S=\partial_{G_F}(x)\cup \partial_{G_F}(y)$, then
	$|S|=5$ and $r^2(S)\leq 3$.
\end{itemize}
By Lemma~\ref{lemma:edgenum}, we are done.



Then suppose that $|R_1|=|R_3|=1$. It follows that  the degree sequence of $G_F$ is $(1,2,2,\ldots,3)$. By Lemma~\ref{lemma:strseq}, $ G_F$ is obtained  from the path
$v_1v_2\cdots v_{\ell}$ by adding $v_iv_{\ell}$ for some $2\leq i\leq \ell-2$. By Lemma~\ref{lemma:subk4+}, $F\notin \rw(n,P_\ell)$, a contradiction.

 Finally, suppose that $|R_1|=0$. Then $ G_F$ is isomorphic to the cycle $C_{\ell}=v_1v_2\cdots v_{\ell}v_1$. Suppose first that $e_1=v_iv_j\in E(\bar{G}_F)$. Then without loss of generality, say that $i=1$ and $3\le j\le \frac{\ell+1}{2}$. Then  by setting $c(v_1v_j)=c^*(v_{\ell-1}v_{\ell-2})$, there is no rainbow path of order  $\ell$ containing $e_1$ in $F+ \{e_1\}$.     Hence, by symmetry, assume that 
$e_1=zw$ where $w\in V(F)\backslash V(G_F)$ and  $z=v_1$. Then by  setting 
$c(e_1)=c^*(v_3v_4)$,  the  longest rainbow path in $F_1$ containing $e_1$ is $v_4v_5\cdots v_\ell v_1w$ which has order  $\ell-1$, a contradiction. 

\medskip
\noindent\textbf{Case 3.} 
 $  n(G_F)=\ell+1$. 

Note that  $ G_F$ is a tree. 
If $4\leq|R_1|\leq \ell-1$, then set $S=\partial_{G_F}(R_1)$. Since
$r^1(S)\leq 2\leq |S|-2$ and $|E(F)\setminus S|\geq 1$,
the ordering is not 
weakly $P_\ell$-rainbow saturated    by Lemma~\ref{lemma:edgenum}.  If $|R_1|=3$, by Lemma~\ref{lemma:substar}, $F\notin  \rw(n,P_\ell)$. If $|R_1|=2$, then $ G_F$ is the path $v_1v_2\cdots v_{\ell+1}$. By Lemma~\ref{lemma:subedge}, $F\notin  \rw(n,P_\ell)$.
\end{proof}

\begin{lemma}\label{lem:disct=1}
Let $F$ be a graph on $n$ vertices, where $n\ge \ell+1\ge 31$. If $ G_F$ is disconnected with $\ell$ edges, then $F\notin  \rw(n,P_\ell)$.
\end{lemma}

\begin{proof}
Suppose that $F\in  \rw(n,P_\ell)$. 
By the same argument as at the beginning of the proof of Lemma~\ref{lemma:conn2}, we have
$\Delta(F)\leq 3$.

Suppose $e_1=zw$. If $z$ and $w$ are  contained in the same component $H$ of $ G_F$, then $|E(H)|=\ell-1$ and $ n(H)=\ell$. By Lemma~\ref{lemma:l1e}, $H=v_1\cdots v_{\ell}$ is a path and $e_1=v_1v_\ell$. Now, suppose that $e_2=pq$. Again by  Lemma~\ref{lemma:l1e} and symmetry, assume that $p=v_i$ with $1\le i\le \frac{\ell+1}{2}$ and $q\in V(H')$ where $H$ and $H'$ are the only two components of $ G_F$. We may assume that $E(H')=\{qq'\}$.  Then by setting $c(e_1)=c^*(v_{\ell-1}v_\ell)$ and $c(e_2)=c^*(v_{i}v_{i+1})$,  the longest rainbow path in $F_2$ containing $e_2$ is $q'qv_iv_{i-1}\cdots v_1v_\ell$, which has order  less than  $\ell $. Therefore, $e_1=zw$ must be an edge 
connecting two components of $ G_F$. 
 
Let $G_1$ and $G_2$ be two components of $ G_F$
such that $z\in V(G_1)$ and $w\in V(G_2)$ with $ n(G_1)\ge  n(G_2)$.   If $d_{G_F}(z),d_{G_F}(w)\geq 2$, then  set
$S=\partial_{G_F}(z)\cup \partial_{G_F}(w)$. Since $4\leq |S|\leq 6$ and $r^1(S)\leq 2$, Lemma~\ref{lemma:edgenum} gives a contradiction. 
If $d_{G_F}(z)=1$, then let $c(e_1)=c^*(e_z)$. It follows that there is no rainbow $P_\ell$ in $F+ \{e_1\}$ containing $e_1$ since $ n(G_1)\ge  n(G_2)$. 
Thus, assume that
$d_{G_F}(z)=2$ and $d_{G_F}(w)=1$. Since there is a  rainbow $P_{\ell}$ containing $e_1$ when we set $c(e_1)=c^*(e_w)$, it holds that $e(G_1)\ge \ell-2$ and the length of the longest path starting at $z$ in $G_1$ is at least $\ell-2$. Consequently, there is a path $Q=v_1v_2\cdots v_{\ell-1}$ in $G_1$ where $v_1=z$. Note that if $e(G_1)=\ell-2$, then $G_1=Q$, which implies that $d_{G_F}(z)=d_{G_F}(v_1)=1$, contradicting the assumption that $d_{G_F}(z)=2$.  Thus, it is sufficient to assume that $e(G_1)=\ell-1$ and $e(G_2)=1$.
Let $x\neq v_2$ be another vertex   adjacent to $z$. If $x\notin V(Q)$, then by setting $c(zw)=c^*(zv_2)$, the  longest rainbow path containing $zw$ 
in $F+ \{zw\}$ is $w'wzx$, where $\{ww'\}=E(G_2)$, a contradiction.  If $x\in V(Q)$,
that is, $x=v_i$ for some $3\leq i\leq \ell-1$. If $i\leq \ell-3$, set $c(zw)=c^*(v_iv_{i+1})$.  If $\ell-2\leq i\leq \ell-1$, set $c(zw)=c^*(v_{\ell-3}v_{\ell-2})$. In both cases, the length of  the longest rainbow path  containing $zw$ 
in $F_1$ is at most $\ell-2$, a contradiction. 
\end{proof}

By the above results, \Cref{main:path} holds.   
\begin{proof}[Proof of \Cref{main:path}]
It follows from \Cref{lem:ulbdd}, \Cref{lemma:l1e}, \Cref{lem:ct=1}, \Cref{lemma:conn2}, and \Cref{lem:disct=1}. 
\end{proof}

\section{Stars}\label{sec:star}
In this section, we determine  the weak rainbow saturation number of stars and show  that the corresponding extremal graphs are not
unique. For $\ell\ge 6$, recall that  $S_\ell$ denotes  the star on $\ell$ vertices.  

\begin{THMMAIN2} 
    Let $n$ and $\ell$ be integers with $n\ge 3\ell^2$ and  $\ell\ge 6$. Then $\s(n,S_{\ell})=\binom{\ell}{2}-1$. Furthermore, the extremal graphs are not unique. 
\end{THMMAIN2}
\begin{proof}
 
We first show that $\s(n,S_\ell)\ge \binom{\ell}{2}-1$. Suppose for contradiction that $F\in \rw(n,S_\ell)$ with $e(F)< \binom{\ell}{2}-1$. 
Let $V(G_F)=\{v_1,\ldots,v_t\}$ with $d(v_1)\ge \cdots\ge d(v_p)\ge \ell-1>d(v_{p+1})\ge \cdots\ge d(v_q)\ge \ell-2>d(v_{q+1})\ge \cdots \ge d(v_t)\ge 1>d(v_{t+1})=\cdots=d(v_n)=0$.

Note that $t\ge \ell-1$. We first show that the following fact holds. 
\begin{claim}\label{claim:v-pq}
The following holds. 
\begin{enumerate}[label=(\arabic*)] 
    \item \label{clp} $p\le \ell-3$. 
    \item \label{clqq} $ q\le \ell$.  
\end{enumerate}
\end{claim}
\begin{clproof}
For  \ref{clp},  suppose that $p\ge \ell-2$. Then 
 \begin{align*}
    &e(F)\ge e(F[\{v_1,\cdots,v_p \}])+\sum_{v\in \{v_1,\cdots,v_p \}} (d(v)-d_{F[\{v_1,\cdots,v_p \}]}(v))\\
 &=\sum_{v\in \{v_1,\cdots,v_p \}} d(v) -  e(F[\{v_1,\cdots,v_p \}])\\
 &\ge (\ell-2)(\ell-1)-\frac{(\ell-2)(\ell-3)}{2}\\
 &=\frac{\ell(\ell-1)}{2}-1, 
\end{align*}   a contradiction.
Thus, $p\le \ell-3$. 

 For \ref{clqq}, suppose that $q\ge \ell+1$. Then $ 
        e(F) \ge  \frac{(\ell+1)(\ell-2)}{2} 
         \ge \frac{\ell(\ell-1)}{2}-1$, a contradiction.
\end{clproof}

Observe that the following  edges are privileged  in $\bar{F}$ and so could be added to $F$ first. 
\begin{enumerate}
    \item[Step 1]  $v_iu$ where $i\in [p]$, $u\in V(F)\setminus \{v_i\}$. 

Since $d_{F+v_iu}(v_i)=\ell$, there is a rainbow $S_\ell$ rooted at $v_i$ containing $v_iu$.

\item[Step 2] $v_iv_j$ where $i,j\in [q]\setminus[p]$ with $i\neq j$. 

It suffices to consider the following. 
If $c(v_iv_j)=c^*(v_iv_k)$ where $v_k\in N(v_i)$, then there is a rainbow $S_\ell$ consisting of $\{v_iv_j\}\cup \{v_jv:v\in N(v_j) \}$. The same argument holds when $c(v_iv_j)=c^*(v_jv_k)$ where $v_k\in N(v_j)$. 
\end{enumerate}

Let $F'$ be the graph  obtained from $F$ by adding all the edges in Step 1 and  Step 2. Then by Claim~\ref{claim:v-pq}, we have that  $F'\neq K_n$ since $n\ge 3\ell^2$. Let $v_iv_j\in E(\bar{F}')$ be the first edge that we could add to $F'$ with $i<j$. 
Then the following claim holds.

\begin{claim}\label{cldegG'}
It holds that $i\in [p+1,q]$ with $d_{F'}(v_i)-d(v_i)\ge \ell-2$ and  $j\in [n]\setminus [q]$.  
\end{claim}
\begin{clproof}
Note that $i,j\notin [p]$ by Step  1. Additionally, for each $k\in [n]\setminus [q]$, it holds that $d_{F'}(v_k)-d(v_k)\le p\le \ell-3<\ell-2$. Let $S_1=\partial_F(v_k)$, and $S_2=\partial_{F'}(v_k)\setminus S_1$. Then $|S_1|=d(v_k)\le \ell-3$ and $|S_2|\le \ell-3$.   
For any $x\in V(F)$ with $xv_k\in E(\bar{F}')$, let $\alpha=\min\{|S_1|,|S_2|\}$. Then for each $f\in S_2'\subseteq S_2$ with $|S'_2|=\alpha$, there is $f'\in S_1$ such that for $f_1\neq f_2\in S_2'$, it holds that $f'_1\neq f'_2$. Let $c(f)=c^*(f')$ for each $f\in S'_2$. Since each rainbow $S_\ell$ rooted at $v_k$  containing $xv_k$ has at most one of $f\in S'_2$ and $f'\in S_1$ and $ |S_1\cup \{xv_k\}\cup S_2| -\alpha \le \ell-2$, there is no rainbow $S_\ell$ containing $xv_k$ in $F'+\{xv_k\}$ rooted at $v_k$.

Therefore, by Step 2, it holds that   $i\in [p+1,q]$ and $j\in [n]\setminus [q]$. Moreover,  the rainbow $S_\ell$s in $F'+ v_iv_j$ containing $v_iv_j$ must be rooted at $v_i$. Now, by taking $x:=v_j$, $v_k:=v_i$, $S_2:=S_2\cup \{v_iv_j\}$, since $|S_1|=\ell-2$,  we  can deduce  that $d_{F'}(v_i)-d(v_i)\ge  \ell-2$.
\end{clproof} 

  Let $$I=\{v_k:d_{F'}(v_k)-d(v_k)\ge \ell-2;k\in [p+1,q]\}.$$ 
  Since $F'\neq K_n$, we have that $I\neq\emptyset$. Consequently, by the proof of \Cref{prop:proA}, all the following edges could be added to $F'$, first.

\medskip

{\bf  Step 3.} $v_iv_j$ where $i\in I$ and $j\in [n]\setminus [q]$ with $i\neq j$.

\medskip

Directly, by  \ref{clqq}  and Claim~\ref{cldegG'}, we have that $q\in \{\ell-1,\ell\}$.  Now, let $v_iv_j$ be the first edge that we could add to $F$ such that $i,j\notin [q]$. Let $F''$ be the   graph obtained from $F$ by adding all the edges before $v_iv_j$.  Equivalently, $F''$ is the graph  obtained by adding all edges in Step  1--3.  
Notice that such an edge exists since $n\ge 3\ell^2$.  

Observe that by the same proof of Claim~\ref{cldegG'}, it holds that either $d_{F''}(v_i)-d(v_i)\ge \ell-2$ or $d_{F''}(v_j)-d(v_j)\ge \ell-2$. Since $i,j\notin [q]$, by Step 1, Step 3 and the requirement that $d_{F''}(v_i)-d(v_i)\ge \ell-2$ or $d_{F''}(v_j)-d(v_j)\ge \ell-2$, we have that $|I|\ge \ell-2-p$. Recall that $q\in \{\ell-1,\ell\}$.  It suffices to consider the following two cases.\\ 

\noindent {\bf Case 1. } $q=\ell-1$. 

Since $|I|\ge \ell-2-p$, by the definition of $I$, $\{v_{p+1},\ldots, v_q\}$ is an independent set of $F$.   Consequently, there is at most one vertex $ v$ in $\{v_{p+1},\ldots, v_q\}$ that has  neighbors in $\{v_1,\ldots, v_p\}$. Let $W=\{v,v_1,\cdots, v_p\}$ and $W'=\{v_{p+1},\ldots, v_q\}\setminus \{v\}$. Therefore,  
\begin{align*}
    e(F)&\ge e(F[W])+ \partial_F(W)+\partial_F(W')\\
    & =e(F[W])+\sum_{v\in W} (d(v)-d_{F[W]}(v))+\sum_{v\in W'} d(v) \\
    &=\sum_{v\in W\cup W'}d(v)-e(F[W])\\
    &\ge  (\ell-1)(\ell-2)-{{\ell-2}\choose 2} \\
   &\ge \frac{\ell(\ell-1)}{2}-1,
\end{align*} 
since $p+1\le \ell-2$ by \ref{clp}, a contradiction. \\

\noindent{\bf Case 2. } $q=\ell$. 

Similarly,  in the graph $F$  all but at most two vertices in $\{v_{p+1},\ldots, v_q\}$  have  at most one neighbor in $\{v_1,\ldots, v_q\}$. 
For simplicity, let $W=\{v_1,\cdots,v_q\}$, $W'=\{v_1,\ldots, v_p\}$, and $W''\subseteq \{v_{p+1},\cdots,v_q\}$ such that each vertex in $W''$ has at least  two neighbors in $W$. 
Then $e(F[W'\cup W''])\le {{p+2}\choose 2}$. Since each vertex in $W\setminus (W'\cup W'')$ has at most 1 neighbor in $W$, it holds that in $F[W]$, the number of edges with one endpoint in $W\setminus (W'\cup W'')$ is at most $\ell-p$. Thus, 
$$\begin{aligned}
    e(F[W])\le  {{p+2}\choose 2}+\ell-p. 
 \end{aligned}
 $$
 Therefore,
 \begin{align*}
    e(F)&\ge e(F[W])+\partial_F(W)\\
 &=e(F[W])+\sum_{v\in W} (d(v)-d_{F[W]}(v))\\
 &=\sum_{v\in W} d(v) -  e(F[W])\\
 &\ge p(\ell-1)+(\ell-p)(\ell-2) -  e(F[W]). 
 \\ &\ge p(\ell-1)+(\ell-p)(\ell-2)-\left({{p+2}\choose 2}+\ell-p\right)\\
 &=\ell^2-3\ell-\frac{1}{2}p^2+\frac{p}{2}-1\\
& \ge \frac{\ell(\ell-1)}{2}-1
\end{align*} 
where the last inequality holds by \ref{clp}, a contradiction. \\

Therefore, $\s(n,S_\ell)\ge \frac{\ell(\ell-1)}{2}-1$. We finish the proof by showing that $|\rw(n,S_\ell)|\ge 2$. Let $G$ be a graph with $V(G)=X\cup Y$ and $\ell+2\le |V(G)|\le 2(\ell-2)$ satisfying all the following.
\begin{itemize}
    \item $G[X]$ forms a clique of size $\ell-2$;
    \item $ Y $  is an independent set of $G$; 
    \item  $d(v)=\ell-1$ where $v\in X$; and
    \item each vertex in $Y$ has at least one neighbor in $X$. 
\end{itemize}
Then $e(G)=\frac{\ell(\ell-1)}{2}-1$. Moreover, the graph $G$ is not unique. 
\begin{claim}\label{clain:exgs}
   Let $F=G\cup (n-|V(G)|)K_1$. Then $F\in \rw(n,S_\ell)$.
\end{claim} 
\begin{clproof}
Note that we first add all the Step 1 edges to $F$.  Denote the resulted  graph by $F^*$. Then $d_{F^*}(v)-d(v)=\ell-2$ for each $v\in V(F)\setminus V(G)$. Now, we can add all the edges in $\bar{F}$ to $F$   of the form $uv$ where $u \in V(F)\setminus V(G)$ and $v\in V(F)$. Let $F^{**}$ be the resulting graph. 
Then  for each vertex $v$ in $Y$, we have that $d_{F^{**}}(v)-d(v)\ge \ell-2$ since $n\ge 3\ell^2$. Thus, we can add  all the edges which have both endpoints in $Y$ and so we obtained the $K_n$.  
\end{clproof}
Therefore, the conclusion holds. 
\end{proof}

\begin{remark}
We have shown that $\ell+1=\s(n,P_\ell)<\s(n,S_\ell)={\ell \choose 2}-1$ for large enough $n$ and $\ell$. 
 It is not known   whether any tree $T$ on $\ell$ vertices satisfies $\s(n,P_\ell)\le \s(n,T)\le \s(n,S_{\ell})$ for some sufficiently large $n$.  Additionally, 
when $\ell\leq 8$, an extremal graph in $\rw(n,P_\ell)$ is not necessarily unique. For example, both $P_3$ and $2K_2$ are weakly  $P_3$-rainbow saturated. When $\ell\geq 30$, although we have constructed an extremal graph, it remains unknown whether it is unique.
\end{remark}

\section{Cycles}\label{sec:cycle}
In this section, we  obtain an 
upper bound for weak saturation numbers of cycles, whose  leading coefficient is strictly smaller than $\frac32$ for large enough $n$, by giving an explicit construction. 
\begin{THMMAIN3}
    For $\ell\ge 4$ and $n\ge 10\ell$, we have that 
    $$
   n+\frac{n}{6\ell}\le  \s(n,C_\ell)\le\frac{n-4\ell}{\ell-1}  \ell +\binom{6\ell}{2}  
    $$
\end{THMMAIN3}
\begin{proof}
For the lower bound, let $F\in \rw(n,C_\ell)$. We assert that $F$ must be connected and $\delta(F)\ge 2$. Suppose not, let $u\in V(F)$ with $d(u)=\delta(F)\le 1$. Let $e=uv\in E(\bar{F})$ be the first edge in $E(\bar{F})$ with an endpoint $u$.
If $d(u)=0$, then there is no cycle containing $uv$. Now suppose that $d(u)=1$ and let $e_u$ be the unique edge incident with $u$. If we set $c(e)=c^*(e_u)$, then there is no rainbow $C_\ell$ containing $e$, a contradiction. If $F$ is a disjoint union of cycles, then let $e=uv$ be the first edge we added that joins two distinct cycles. Consequently, there is no rainbow cycle contains $uv$, a contradiction. 
Thus, $\delta(F)\ge 2$ and $F$ is connected.   Then we consider a path whose internal vertices all have
degree 2 in $F$. Let $P=w_0w_1\cdots w_tw_{t+1}$ be such a path, i.e., $d(w_i)=2$ for each $i\in [t]$ and $d(w_0),d(w_{t+1})\ge 3$. Now, we claim that $t\le 2\ell-1$. Suppose that $t=2\ell$.  Let $e=xw_i\in E(\bar{F})$ be the first edge in $E(\bar{F})$ with an endpoint in $\{w_1,\ldots,w_t\}$. By symmetry, it suffices to consider that $i\in [\ell]$. Then   when $c(xw_i)=c^*(w_{i-1}w_i)$, there is no rainbow $C_\ell $ containing $xw_i$, a contradiction. 
  Let $n_{\ge3}$ be the number of degree at least 3 vertices in $F$. Let $F^*$ be the (multi)graph obtained from $F$ by contracting all the vertices of degree 2.    Then \begin{align*}
      e(F^*)&\ge  \max\left\{\frac{3}{2}n_{\ge 3}, \frac{n-n_{\ge 3}}{2\ell-1}\right\}, \text{~and ~so~} \\
      e(F)&=(n-n_{\ge 3})+ e(F^*)\ge n +\frac{n}{6\ell}.
  \end{align*}

We now consider the upper bound, which is obtained by an
explicit construction. Let $$ t = \lfloor \frac{n - 4\ell}{2\ell-2} \rfloor   \text{ and }  q= n- (2\ell-2)t.$$  Then $4\ell\le q\le 6\ell-2$. 
 We define a graph $F$ as the one-point union of a complete graph $K_q$ and $t$ copies $ H_1, H_2, \dots, H_t $ of a given graph $H$ at $v$, where $$\begin{aligned}V(H) &= \{v, w_1, \dots, w_{2\ell-2}\};\\  E(H) &= \{ v w_1, v w_2,   v w_{2\ell-2} \} \cup \{ w_i w_{i+1} : 1 \le i \le 2\ell-3 \}.
 \end{aligned}$$
 Then $e(F)=2t\ell+{q\choose 2}$ when $\ell\ge 4$. 
Equivalently, let $V(K_q)=\{v=u_1,\dots, u_q\}$. Then $$V(F)=V(K_q)\cup \bigcup^{t}_{i=1}V(H_i) \text{ and } E(F)=E(K_q)\cup  \bigcup^{t}_{i=1}E(H_i),$$
 where $V(K_q)\cap V(H_i)= V(H_i)\cap V(H_j)=\{v\}$ for $i\neq j\in [t]$. 

   For each $ a \in [t] $, let $ w_i^a $ denote the copy of $ w_i $ in $ H_a $. Then $ w_i^a \in V(H_a) $. By writing $ w_i $ without a superscript, we mean that the operation is performed on all copies $ H_1, H_2, \dots, H_t $.     Since $K_q$ is a complete graph and $ n(K_q) \ge 4\ell$, the following fact holds.  
    \begin{claim}
        \label{cl:fobedges}
       Let $K_p\subseteq K_q$ with $p=2\ell$ and $ X \subseteq E(K_q) $ with $ |X| \le 3 $. Then for any two distinct vertices $ x, y \in V(K_p) $, there exists an $ (x,y) $-path of order $ k $ in $ K_p - X $ where $3\le k \le \ell-1$.
    \end{claim}
    \begin{clproof}
      Without loss of generality, assume that $|X|=3$, $k=\ell-1$.  Since $n(X)\le 6$ and $q\ge 2\ell$, if $\ell\ge 5$, we are done. Thus, assume that $\ell=4$. Then $k=3$. For $ x, y \in V(K_p) $, since $|X|\le 3$ and $p\ge 8$, there exists $z\in V(K_p)\setminus \{x,y\}$ such that $\{zx,zy\}\cap X=\emptyset$. 
    \end{clproof}
    
It is sufficient to show that $F$ is weakly   $C_\ell$-rainbow saturated.  
To make the  proof readable, we divide the proof into two cases  according to
whether $\ell\ge 5$ or $\ell=4$. 

\medskip

\noindent{\bf Case 1.} $\ell\ge 5$.

 We add edges in $\bar{F}$ to $F$ according to the following order. When adding some edge $e$ to $F$, without loss of generality, it suffices to assume that $c(e)\in c^*(E({F}))$. 

    \begin{enumerate}[label=(S1.\arabic*)]
    
    \item\label{step1-2} For any $u\in V(K_q)\setminus \{v\}$, add $uw_1$ and $uw_2 $    to $E(F)$ in order. 

 For $uw_1$, if $c(uw_1)\in c^*(E (K_q ))$, let $X=\{e\in E(K_q):c(uw_1)=c^*(e)\}$. By \Cref{cl:fobedges}, there is a  $(u,v)$-path $P$ of order $\ell-1$ in $K_q$ avoiding $X$. Then $uPvw_1u$ is a rainbow $C_\ell$. Thus, assume that $c(uw_1)=c^*( vw_1 )$ or $c(uw_1)\in c^*(\{vw_2,w_1w_2\})$.  Then $uw_1w_2v$ or $uw_1v$ is rainbow  whose set of colors is disjoint with $c^*(E(K_q))$, respectively. By  \Cref{cl:fobedges} with $X=\emptyset$  and $F$ being rainbow-colored, there are rainbow $(u,v)$-paths of order  $\ell-2$ and $\ell-1$, respectively. By symmetry, the same holds for $uw_2 $.  
    \item\label{step1-1} Add  $vw_\ell$, $uw_\ell$ for each $u\in V(K_{q})\setminus \{v\}$, $w_1w_\ell$ and $w_2w_\ell$ to  $E(F)$.  
    
    For $vw_{\ell}$, at least one of $vw_2w_3\cdots w_{\ell}v$ and $vw_{\ell}w_{\ell+1}\cdots w_{2\ell-2}v$ must be a rainbow $C_\ell$. 

    For $uw_\ell$, by \Cref{cl:fobedges}, for any assignment of $c(uw_\ell),c(v{w_\ell})$ there is a  $(u,v)$-path $P$ of order $\ell-1$ in $K_q$  such that $\{c(uw_\ell),c(v{w_\ell})\}\cap c^*(E(P))=\emptyset$. Then $uw_\ell vPu$ is a rainbow $C_\ell$. 
 
  For $w_1w_\ell$, let distinct $u_1,u_2\in V(K_q)\setminus\{v\}$.  By \Cref{cl:fobedges}, for any assignment of $c(u_1w_1),c(u_2{w_\ell}),$ there is a  $(u_1,u_2)$-path $P$ of order $\ell-2\ge 3$ in $K_q$  such that $\{c(u_1w_1),c(u_2{w_\ell}),c(w_1w_\ell)\}\cap c^*(E(P))=\emptyset$. Then $u_1w_1w_\ell u_2Pu_1$ is a rainbow $C_\ell$. The same holds for $w_2w_\ell$.   

    \item\label{step1-3} For $i\in [\ell-1]\setminus\{1,2\}$ in increasing order,  add $w_1w_i$,   $uw_i$ for each $u\in V(K_q) \setminus \{v\}$, $vw_i$ and $w_j{w_i}$ for $j\in [i-2]\setminus \{1\}$  to $E(F)$.  

First,  consider $w_1w_i$.  We prove by induction on $i$  in increasing order. 

Assume first that $i=3$.  If $c(w_1w_i)\notin c^*(E(w_iw_{i+1 }\cdots w_\ell))$, since $q\ge 4\ell$, for any assignment of $c$, there is $x\neq v\in V(K_q)$ such that $c(x w_1),c(xw_\ell)\notin c^*(E(w_iw_{i+1 }\cdots w_\ell))$. Thus, $xw_1w_iw_{i+1 }\cdots w_\ell x$ is a rainbow $C_\ell$. Hence, assume that $c(w_1w_i)\in c^*(E(w_iw_{i+1 }\cdots w_\ell))$. If $\ell\ge 6$, since $q\ge 4\ell$, similarly, there exists $x\in V(K_q)\setminus \{v\}$ such that $c(xw_{i-1})\notin c^*(\{w_iw_{i-1 },vw_1\}) $. By \Cref{cl:fobedges}, there is an $(x,v)$-path $P$ of order $\ell-3\ge 3$ in $K_q$  such that $\{c(w_1w_i),c(xw_{i-1 })\}\cap c^*(E(P))=\emptyset$. Then $xw_{i-1 }w_iw_1vPx$ is a rainbow $C_\ell$.  If $\ell=5$, since $q\ge 4\ell$,  there are distinct $x,y\in V(K_q)\setminus \{v\}$ such that $c^*(w_{i-1 }w_i)\notin c(\{xw_1,xw_{i-1},yw_1,yw_{i-1}\})$. Then at least one of  $xyw_1w_iw_{i-1 }x$  and $yxw_1w_iw_{i-1 }y$ is rainbow.   

Assume  that  $ i=4$.  If $c(w_1w_i)\notin c^*(E(w_iw_{i+1}\cdots w_\ell))$,  since $q\ge 4\ell$, similarly, there are $x,y\in V(K_q)\setminus\{v\}$ such that $xw_1w_4w_5\cdots w_\ell yx$ or $yw_1w_4w_5\cdots w_\ell xy$ is rainbow.   If $c(w_1w_i)\in c^*(E(w_iw_{i+1}\cdots w_\ell))$, the same proof for the case that $i=3$ works here.

For $ 5\le i\le \ell-1$, if $c(w_1w_i)\notin c^*(E(w_iw_{i+1}\cdots w_\ell))$, since $q\ge 4\ell$,   there  are distinct  $x,y \in V(K_q)\setminus\{v\}$ such that  $c(xw_\ell),c(yw_1)\notin c^*(E(w_iw_{i+1}\cdots w_\ell))$. Again, since $q\ge 4\ell$, by \Cref{cl:fobedges}, there is an $(x,y)$-path $P$ of order $ (i-2)\ge 3$ in $K_q$ such that $c^*(E(P))\cap \{c(xw_\ell),c(yw_1),c(w_1w_i)\}=\emptyset$. Then $yw_1w_iw_{i+1}\cdots w_\ell xPy$ is a rainbow $C_\ell$. Thus, assume that  $c(w_1w_i)\in c^*(E(w_iw_{i+1}\cdots w_\ell))$. Then the same proof for $i=3$ works here.

Secondly, for  $uw_i$ where $u\in V(K_q)$ (first consider some $u\neq v$), fix any $u'\in V(K_q)\setminus \{v,u\} $. Since $|\{u'w_1,w_1w_i,uw_i\}|\le 3$, by \Cref{cl:fobedges}, for any assignment of $c$, there is a  $(u',u)$-path $P$ of order $\ell-2\ge 3$ in $K_q$ such that $c(\{u'w_1,w_1w_i,uw_i\})\cap c^*(E(P))=\emptyset$. Then $u'w_1w_iuPu'$ is a rainbow $C_\ell$. 

For $w_jw_i$ where $2\le j\le i-2$, fix any distinct $x,y\in V(K_q)\setminus \{v\} $. Since $|\{xw_j,w_jw_i,yw_i\}|\le 3$, by \Cref{cl:fobedges}, for any assignment of $c$,    there is an $(x,y)$-path $P$ of order $\ell-2\ge 3$ in $K_q$ such that $c(\{xw_j,w_jw_i,yw_i\})\cap c^*(E(P))=\emptyset$. Then $xw_jw_iyPx$ is a rainbow $C_\ell$.

    \item\label{step1-4} For $i\in [\ell+1,2\ell-2]$ in descending order, add $w_iw_{i+1-\ell}$,  $uw_i$ for each $u\in V(K_q) \setminus \{v\}$, $vw_i$ (if exists) and $w_k{w_i}$ for $k\in [i-\ell]$  to $E(F)$.  

    First,  consider $w_iw_{i+1-\ell}$.  We prove by induction on $i $ in descending order. 
    
    Note that $i+1-\ell<\ell$.  Since $w_{i+1-\ell}w_{i+2-\ell}\cdots w_{i-1}w_iw_{i+1-\ell}$ is of order $\ell$, we may assume that $c(w_iw_{i+1-\ell})\in c^*(E(w_{i+1-\ell}w_{i+2-\ell}\cdots w_{i-1}w_i)) $. 

Consider the base case that $i=2\ell-2$. Since $q\ge 4\ell$, there is $u\in V(K_q)\setminus\{v\}$ such that $c(u w_{i+1-\ell} )\neq c^*(vw_{i})$ and there is a  $(u,v)$-path $P$ of order $\ell-2\ge3$ with $c^*(E(P))\cap c(\{w_iw_{i+1-\ell},uw_{i+1-\ell} \})=\emptyset$. Then $uw_{i+1-\ell}w_ivPu$ is a rainbow $C_\ell$. 

Consider general $\ell+1\le i\le 2\ell-3$.  Since $c(w_iw_{i+1-\ell})\in c^*(E(w_{i+1-\ell}w_{i+2-\ell}\cdots w_{i-1}w_i)) $ and $q\ge 4\ell$, by replacing the path $w_1w_iw_{i-1}$ in \ref{step1-3} with $w_{i+1-\ell}w_iw_{i+1}$, the same result holds here.

Then for $uw_i$  where $u\in V(K_q)  $ and $w_k{w_i}$ where  $k\in [i-\ell]$  to $E(F)$,  by replacing  $w_1w_i $ in \ref{step1-3} with $w_{i+1-\ell}w_i $ or $w_kw_i$, the same result holds here.

    \item\label{step1-5} For any $i,j\in [2\ell-2]$ and $a,b \in [t]$, add $w_i^{(a)}w_j^{(b)}$ as long as it has not been added to $E(F)$.\\
  Take any distinct  vertices $u_1,u_2 \in V(K_q)\setminus \{v\}$. Since $|c(E(u_1w_i^{(a)}w_j^{(b)}u_2))\cap c^*(E(K_q))|\le 3$, there is a rainbow $(u_1,u_2)$-path $P$ in $K_q$ of order $\ell-2\ge3$ with   $c^*(P)\cap c(E(u_1w_i^{(a)}w_j^{(b)}u_2))=\emptyset$  by  \Cref{cl:fobedges}. It follows that $u_1w_i^{(a)}w_j^{(b)}u_2P u_1$ is a rainbow $C_\ell$.
\end{enumerate}
Observe that all the edges in $\bar{F}$ have been added to $F$. 
\medskip

Since $\ell-2<3$ when $\ell=4$, some of the above arguments do  not work for $\ell=4$.  
\medskip

\noindent{\bf Case 2.} $\ell=4$.  

Similar to the previous case, we add edges in $\bar{F}$ to $F$ according to the following order. 

    \begin{enumerate}[label=(S2.\arabic*)]
    \item\label{step2-1} Add $vw_4$ and then $uw_1,uw_2,uw_4$ for each $u\in V(K_q)\setminus\{v\}$. 

By \ref{step1-1} for $vw_4$ and  $uw_4$, we are done. For $uw_1 $, if $c(uw_1)\neq c^*(vw_1)$, by  \Cref{cl:fobedges}, then there is  a $(u,v)$-path $P$ of order $\ell-1$ in $K_q$ such that $c(uw_1)\notin c^*(E(P))$. Then $uw_1vPu$ is a rainbow $C_4$. Thus, assume that  $c(uw_1)=c^*(vw_1)$. Then $uw_1w_2vu$ is a rainbow $C_4$. The same argument works for $uw_2$.

\item\label{step2-2}  Add $w_1w_4,w_2w_4,w_1w_3$,   $uw_3$ for each $u\in V(K_q)\setminus \{v\}$ and $vw_3$. 

For $w_1w_4$, if $c(w_1w_4)\notin  c^*(E(w_1w_2w_3w_4))$, then $w_1w_2w_3w_4w_1$ is a rainbow $C_4$. Thus, assume that $c(w_1w_4)\in  c^*(E(w_1w_2w_3w_4))$. Pick any distinct $u_1,u_2\in V(K_q)\setminus \{v\}$, then $u_1w_4w_1u_2u_1$   or $u_1w_1w_4u_2u_1$ is rainbow. 

For $w_2w_4$, since $q\ge 8$, there exist distinct $u_1,u_2\in V(K_q)\setminus \{v\}$ such that $c^*(u_1u_2)\neq c(w_2w_4)$. Then similarly $u_1w_4w_2u_2u_1$  or $u_1w_2w_4u_2u_1$ is rainbow.

For $w_1w_3$, note that there exists $u\in V(K_q)\setminus \{v\}$ such that $c(uw_1),c(uw_2),c(uw_4)\notin \{c^*({w_2w_3}),$ $c^*(w_3w_4)\}$. Then either $uw_1w_3w_2u$ or $uw_1w_3w_4u$ is rainbow. 

For $uw_3$ where $u\in V(K_q)\setminus \{v\}$, $uw_3w_1w_4u$ is rainbow since  each edge is contained in $E(\bar{F})$. 

For $vw_3$,  by  \ref{step1-1}, we are done. 

\item\label{step2-3}  Add $w_2w_5,w_1w_5,w_3w_5$,
  $uw_5$ for each $u\in V(K_q)$.   

For $w_2w_5$, either $w_2w_3w_4w_5w_2$ or $vw_2w_5 w_6v$ is rainbow, since only $w_2w_5\in E(\bar{F})$. 

For $w_1w_5$, let $u\in V(K_q)\setminus \{v\}$. Then $uw_1w_5w_2u$ is rainbow since all the edges are contained in $E(\bar{F})$.

For $w_3w_5$, let $u\in V(K_q)\setminus \{v\}$. Then $uw_3w_5w_1u$ is rainbow since all the edges are contained in $E(\bar{F})$.

For $u\in V(K_q)$, $uw_3w_1w_5u$ is rainbow since all the edges are contained in $E(\bar{F})$. 

\item\label{step2-4}  Add $w_3w_6$, $uw_6$ for each $u\in V(K_q)\setminus \{v\}$,   $w_1w_6$, $w_2w_6$, and $w_4w_6$. 

 For $w_3w_6$, either $w_3w_4w_5w_6w_3$ or $vw_2w_3 w_6v$ is rainbow, since only $w_3w_6\in E(\bar{F})$. 

For $uw_6$ where $u\in V(K_q)\setminus \{v\}$, $uw_6w_3w_1u$ is rainbow since all the edges are contained in $E(\bar{F})$. 

For $w_1w_6,w_2w_6$ and $w_4w_6$ in order, let $u\in V(K_q)\setminus \{v\}$. We have that $uw_1w_6w_3u$, $uw_2w_6w_3u$, and $uw_4w_6w_3u$ are rainbow, respectively. 
 \item\label{step2-5} For any $i,j\in [2\ell-2]$ and $a,b \in [t]$, add $w_i^{(a)}w_j^{(b)}$ to $E(F)$.

 Note for $c(w_i^{(a)}w_j^{(b)})$, there are distinct $u_1,u_2\in V(K_q)\setminus \{v\}$ such that $c^*(u_1 u_2)\neq c(w_i^{(a)}w_j^{(b)})$. Then either $u_1w_i^{(a)}w_j^{(b)}u_2u_1 $ or $u_2w_i^{(a)}w_j^{(b)}u_1u_2 $ is rainbow. 
    \end{enumerate}

\medskip

Therefore, the conclusions hold. 
\end{proof} 

For $\s(n,C_3)$,   it was asked whether there exists a constant $c$ such that $\s(n, C_3)= 2n+c$  \cite{Li2025}. We prove the existence of such a constant
under the assumption that a bipartite
extremal graph exists.   
\begin{proposition}
For an integer $n$ with $n\ge 3$, if there is a bipartite graph $F$ in $\rw(n,C_3)$, then $F=K_{2,n-2}$ and so $\s(n,C_3)=2(n-2)$.
\end{proposition}
\begin{proof}
Let $F=(X,Y)\in \rw(n,C_3)$ be a bipartite graph.  Suppose $F\neq K_{2,n-2}$. Let $e_1,\ldots, e_m$ be a weakly  rainbow saturated ordering of $\bar{F}$. Denote by $e=uv$  the first edge in $e_1,\ldots, e_m$ of the form $u\in X$ and $v\in Y$.  Since $F\in \rw(n,C_3)$, the triangles that contain  $uv$ are of the form $uvxu$ where $uv,ux\in E(\bar{F})$ and $ vx\in E(F)$ by the choice of $e$ or of the form $uvyu$ where $uv,vy\in E(\bar{F})$ and $ uy\in E(F)$ by the choice of $e$.  However, there is no  rainbow $C_3$ created by adding $e$ when $c(ux)=c^*(vx)$ and $c(vy)=c^*(uy)$. Thus, $F=K_{t,n-t}$. 
 One can verify that $K_{t,n-t}$ is weakly $C_3$-rainbow saturated for $n-2\ge t\ge 2$. 
Since $F\in \rw(n,C_3)$, it follows that $F=K_{2,n-2}$.  Thus,  $\s(n,C_3)=2(n-2)$.
\end{proof}

\section*{Acknowledgments}
Bo, Liu and Lian are supported by Fundamental and Interdisciplinary Disciplines Breakthrough Plan of the Ministry of Education of China (JYB2025XDXM207).

\end{document}